\documentclass{amsart}
\usepackage{xcolor}
\usepackage{amsmath,amsfonts,amssymb,color,hyperref, enumerate,tabularx}

\usepackage{amsmath, amsfonts,amsthm,amssymb,amscd, verbatim,graphicx,color,multirow,tikz,tikz-cd,booktabs, caption, mathdots,bm,chngcntr,adjustbox,longtable%,enumitem
}
\usepackage{tikz-cd}
\usetikzlibrary{positioning}
\newtheorem{theorem}{Theorem}[section]
\newtheorem{lemma}[theorem]{Lemma}
\newtheorem{proposition}[theorem]{Proposition}

\newtheorem{definition}[theorem]{Definition}

\newtheorem{remark}[theorem]{Remark}

\newcommand{\agl}{\mathop{\mathrm{AGL}}}

\begin{document}

\title[The density of the Sylow $p$-numbers]{On the density of Sylow numbers}

\author[A.~Lucchini]{Andrea Lucchini}
\address{Universit\`a di Padova, Dipartimento di Matematica
``Tullio Levi-Civita'', Via Trieste 63, 35121 Padova, Italy}
\email{lucchini@math.unipd.it}

\author[P.~Spiga]{Pablo Spiga}
\address{Dipartimento di Matematica e Applicazioni, University of Milano-Bicocca, Via Cozzi 55, 20125 Milano, Italy} 
\email{pablo.spiga@unimib.it}
\begin{abstract}
	Let $p$ be a prime number. We say that a positive integer $n$ is a
	\textit{\textbf{Sylow $p$-number}} if there exists a finite group having
	exactly $n$ Sylow $p$-subgroups. When $p=2$, every odd integer is a
	Sylow $2$-number. In contrast, when $p$ is odd, there exist two positive constants $c_p$ and $c_p^\prime$ such that, denoting by $\beta(p,x)$ the number of Sylow $p$-numbers less than or equal to $x$, 
	\[c_p\,x(\log x)^{\frac{1}{p-1}-1}
	\leq \beta(p,x)\leq 
	c_p^\prime\,x(\log x)^{\frac{1}{p-1}-1}.
	\]
	Moreover if  $\beta_s(p,x)$ is the number of positive integers $n\le x$  such that $n$
	is the Sylow $p$-number of some finite solvable group
	then $$\beta_s(p,x)\sim  c_p\,x(\log x)^{\,\frac{1}{p-1}-1}
	\qquad\text{as } x\to\infty.$$
	In particular, when $p$ is odd, the natural density of Sylow $p$-numbers
	is $0$.
\end{abstract}


\maketitle
\section{Introduction}

In view of Sylow's theorems, fixing a prime $p$, it is natural to ask
whether every positive integer $n \equiv 1 \pmod p$ is a
\textit{\textbf{Sylow $p$-number}}, that is, whether $n$ is the number
of Sylow $p$-subgroups of some finite group.
Every odd integer $n$ is a Sylow $2$-number for the dihedral group of
order $2n$.
For odd primes $p$, however, the question is more delicate.

Philip Hall~\cite{phall} observed that in solvable groups the prime
factorization of a Sylow $p$-number
\[
n = p_1^{a_1}\cdots p_s^{a_s}
\]
satisfies $p_i^{a_i}\equiv 1 \pmod p$ for $i = 1,\dots, s$.
For example, no solvable group has exactly six Sylow $5$-subgroups. Although six is a Sylow $5$-number because the alternating group of degree 5 has six Sylow $5$-subgroups.

About forty years later, Marshall Hall~\cite{mhall} reduced the
determination of the Sylow $p$-numbers to simple groups.
More precisely, he showed that every Sylow $p$-number is a product of
prime powers $q^t \equiv 1 \pmod p$ and Sylow $p$-numbers of non-abelian
simple groups.
Conversely, every such product is in fact a Sylow $p$-number, which can
be seen by taking suitable direct products of affine groups and simple
groups.

Indeed, assume
\[
n=p_1^{a_1}\cdots p_t^{a_t}\cdot s_1\cdots s_u,
\]
where $p_i^{a_i}\equiv 1 \pmod p$ for $i = 1,\dots, t$, and $s_j$ is the
Sylow $p$-number of a non-abelian simple group $S_j$.
Then $n$ is the Sylow $p$-number of the direct product
\[
\agl(1,p_1^{a_1})\times \cdots \times \agl(1,p_t^{a_t})
\times S_1\times \cdots \times S_u.
\]

Using the previous result, Marshall Hall proved that not every positive
integer $n \equiv 1 \pmod p$ is a Sylow $p$-number.
We call such integers \textit{\textbf{pseudo Sylow $p$-numbers}}.
We mention two results from~\cite{mhall}:
\begin{enumerate}
	\item If $n=1+rp$, with $1<r<\frac{p+3}{2}$, then $n$ is a pseudo
	Sylow $p$-number, unless $n$ is a prime power, or
	$r=\frac{p-3}{2}$ and $p>3$ is a Fermat prime.
	\item If $n=1+3p$ and $p\geq 7$, then $n$ is a pseudo Sylow
	$p$-number.
\end{enumerate}
To obtain these results, Marshall Hall used Brauer’s sophisticated
theory of $p$-blocks of defect $1$.
In~\cite{sambale}, B.~Sambale gives an elementary proof of the fact that
if $n=1+2p$, then $n$ is a pseudo Sylow $p$-number unless $n$ is a prime
power.

In this paper we address the question of whether, for a given odd prime $p$,
there exist infinitely many pseudo Sylow $p$-numbers.

\begin{definition}
Let $p$ be an odd prime and let $x$ be a positive real number. We define
$\beta(p,x)$ (respectively, $\beta_s(p,x)$) to be the number of positive integers $n\le x$ such that $n$
is the Sylow $p$-number of some finite (respectively, finite solvable) group.
\end{definition}

We prove the following theorem.

\begin{theorem}\label{thrm}
For every odd prime $p$, there  exist two constants $c_p$ and $c_p'$ such that
$$\beta_s(p,x)\sim  c_p\,x(\log x)^{\,\frac{1}{p-1}-1}
\qquad\text{as } x\to\infty$$
and
$$\beta_s(p,x)\le \beta(p,x)\le c_p'x(\log x)^{\,\frac{1}{p-1}-1}
\qquad\text{when }  x\ge 2.$$
\end{theorem}
Unless explicitly specified, logarithms are assumed to be base $e$.
The constant $c_p$ can be explicitly computed; see Remark~\ref{remark1}.

Clearly, the number $\alpha(p,x)$ of integers $n$ less than or equal to $x$ with $n \equiv 1 \pmod p$ is  1+$\lfloor (x-1)/p\rfloor.$ Hence $$\frac{\beta(p,x)}{\alpha(p,x)} \leq \frac{c_p'}{p}(\log x)^{\,\frac{1}{p-1}-1} \qquad\text{when } x \geq 2.$$
Therefore, when $p$ is odd, not only do infinitely many pseudo Sylow
$p$-numbers exist, but asymptotically very few integers congruent to
$1 \pmod p$ arise as the number of Sylow $p$-subgroups of a finite group.

\section{Solvable groups}\label{sec:thrm}
In this section we only deal with solvable groups.
Before proceeding with the proof of Theorem~\ref{thrm} for solvable groups, we need a deep theorem of Landau~\cite[Section~183, Band~2]{landau1}. Recall that $f:\mathbb{N}\to\mathbb{R}$ is \textit{\textbf{arithmetic multiplicative}} if $f(ab)=f(a)f(b)$, for any two relatively prime natural numbers, and $f(1)=1.$
\begin{theorem}\label{landauu}
Let $p$ be a prime number and let $g_p:\mathbb{N}\to\mathbb{R}$ be the arithmetic multiplicative function defined by $g_p(q^k)=1$ when $q^k$ is a prime power and $p\mid q-1$ and $g_p(q^k)=0$ otherwise. Then $$\sum_{n\le x}g_p(n)\sim C_px(\log x)^{\frac{1}{p-1}-1},$$
as $x\to\infty$   where  $C_p$ is the constant
$$\frac{1}{\Gamma(1/(p-1))}\prod_{q\, \mathrm{ prime }}\left(1-\frac{1}{q}\right)^{\frac{1}{p-1}}\prod_{\substack{q\, \mathrm{prime}\\q\equiv 1\pmod p}}\left(1-\frac{1}{q}\right)^{-1}.$$
\end{theorem}

Let $f,g:\mathbb{N}\to \mathbb{R}$ be functions.
The \textit{\textbf{convolution}} $f\ast g$ is defined by
\[
(f\ast g)(n)=\sum_{d\mid n} f(d)g(n/d).
\]

Let $p$ be a prime number.
Let $f_p:\mathbb{N}\to\mathbb{R}$ be the arithmetic multiplicative
function defined on prime powers $q^k$ by
\begin{equation}\label{eq:1}
f_p(q^k)=
\begin{cases}
1 & \text{if } p\mid q^k-1,\\
0 & \text{otherwise}.
\end{cases}
\end{equation}

Given a prime $q\ne p$, let ${\bf o}_p(q)$ be the smallest positive
integer $\ell$ such that $p\mid q^\ell-1$.
Observe that ${\bf o}_p(q)$ is the order of the congruence class
$q\pmod p$ in the multiplicative group of the field
$\mathbb{Z}/p\mathbb{Z}$.

Similarly, let $h_p:\mathbb{N}\to\mathbb{R}$ be the arithmetic
multiplicative function defined on prime powers $q^k$ by
\[
h_p(q^k)=
\begin{cases}
1 & \text{if } {\bf o}_p(q)\mid k \text{ and } {\bf o}_p(q)>1,\\
0 & \text{otherwise}.
\end{cases}
\]

A positive integer $n$ is \textit{\textbf{squareful}} if every prime divisor of $n$
occurs with exponent at least $2$ in its prime factorization. Observe that the support of $h_p$ is on squareful numbers.

\begin{lemma}\label{l:0}
We have $f_p=h_p\ast g_p$.
\end{lemma}

\begin{proof}
Since $f_p$, $g_p$, and $h_p$ are arithmetic multiplicative functions,
it suffices to show that
\[
f_p(q^k)=(h_p\ast g_p)(q^k)
\]
for every prime $q$ and every positive integer $k$.
We have
\[
(h_p\ast g_p)(q^k)=\sum_{i=0}^k h_p(q^i)g_p(q^{k-i}).
\]

If $p$ is relatively prime to $q-1$, then by definition of $g_p$ we have
\[
(h_p\ast g_p)(q^k)=h_p(q^k)g_p(1)=h_p(q^k),
\]
and $h_p(q^k)=f_p(q^k)$ by the definitions of $f_p$ and $h_p$.

Similarly, if $p$ divides $q-1$, then ${\bf o}_p(q)=1$ and hence
\[
(h_p\ast g_p)(q^k)=\sum_{i=0}^k h_p(q^i)=h_p(1)=1=f_p(q^k).\qedhere
\]
\end{proof}

\begin{theorem}\label{thrms}
For every odd prime $p$, there exists a constant $c_p>0$ such that
\[
\beta_s(p,x)\sim c_p\,x(\log x)^{\,\frac{1}{p-1}-1}
\qquad\text{as } x\to\infty.
\]
\end{theorem}

\begin{proof}
By the main result in~\cite{phall}, $\beta_s(p,x)$ is the cardinality of the
set $B(p,x)$ of the positive integers $m\le x$ such that
\[
m=p_1^{\alpha_1}\cdots p_t^{\alpha_t},
\]
where $p_1,\ldots,p_t$ are distinct primes and
$p_i^{\alpha_i}\equiv 1 \pmod p$ for each $i$.

We first give the argument when $p=3$, since in this case the argument is
much cleaner and only requires more elementary number theoretic results, and then we return to the general case.  Let $B'(3,n)$ be the set of positive integers $m\le n$ such that
$m=x^2+3y^2$ for some $x,y\in\mathbb{Z}$.
First we observe that $B'(3,n)$ is closed under multiplication, that is,
if $m,m'\in B'(3,n)$ and $mm'\le n$, then $mm'\in B'(3,n)$.
Indeed, by hypothesis,
$m=x^2+3y^2$ and $m'=x'^2+3y'^2$ for some $x,y,x',y'\in\mathbb{Z}$, and
\[
mm'=(xx'-3yy')^2+3(xy'+x'y)^2.
\]
Therefore, $mm'\in B'(3,n)$.

Let $m\in B'(3,n)$ and let
$m=p_1^{\alpha_1}\cdots p_t^{\alpha_t}$ be the prime power factorization
of $m$.  
We prove by induction on $m$ that  $m\in B(3,n).$ First we claim that if $p_i\equiv 2 \pmod 3$, then $p_i$ divides both $x$ and $y$. In fact, assume that $\gcd(x,p_i)=1$.
Since $m = x^2+3y^2\equiv 0 \pmod{p_i}$,
$-3$ is a square in $\mathbb{Z}/p_i\mathbb{Z}$.
From the law of quadratic reciprocity, we would deduce
\begin{align*}
	1
	&=\left(\frac{-3}{p_i}\right)
	= (-1)^{\frac{p_i-1}{2}}\left(\frac{3}{p_i}\right) \\
	&= (-1)^{\frac{p_i-1}{2}+\frac{(p_i-1)(3-1)}{4}}
	\left(\frac{p_i}{3}\right)
	= \left(\frac{p_i}{3}\right)=\left(\frac{2}{3}\right)=-1.
\end{align*}
This contradiction proves our claim. Hence, if $p_i\equiv 2\pmod 3$, then $x$ and $y$ are both
divisible by $p_i$, and therefore $(x/p_i)^2+3(y/p_i)^2=m/p_i^2 \in B^\prime(3,n).$ Thus $m/p_i^2\in B(3,n),$ and consequently $m\in B(3,n).$

Conversely, let $m\in B(3,n)$ and, as usual, let
$m=p_1^{\alpha_1}\cdots p_t^{\alpha_t}$ be the prime factorization of
$m$.
If $\alpha_i$ is even, then clearly $p_i^{\alpha_i}\in B'(3,n)$.
When $\alpha_i$ is odd, we have $p_i\equiv 1\pmod 3$, because
$p_i^{\alpha_i}\equiv 1\pmod p$.
Therefore, by Fermat's theorem~\cite[page~1]{Cox},
$p_i\in B'(3,n)$.
Hence $m\in B'(3,n)$.

We have shown that $B(3,n)=B'(3,n)$.
Generalizing a famous result of Landau~\cite{landau}, Bernays~\cite{bernays}
proved that
\[
|B(3,n)|=|B'(3,n)|\sim \kappa\frac{n}{\sqrt{\log n}}=n(\log n)^{\frac{1}{3-1}-1},
\]
for some absolute constant $\kappa$.

\medskip

We now assume that $p$ is an arbitrary odd prime. Observe that we may write
\begin{align*}
\beta_s(p,x)&=\sum_{n\le x}f_p(n),
\end{align*}
where $f_p$ is defined in~\eqref{eq:1}.

Thus, from Lemma~\ref{l:0}, we have
$$\sum_{n\le x}f_p(n)=\sum_{n\le x}(h_p\ast g_p)(n)=\sum_{n\le x}h_p(n)\sum_{m\le x/n}g_p(m).$$
Using Theorem~\ref{landauu}, we may replace the inner sum with $$C_p\frac{x}{n}\left(\log \left(\frac{x}{n}\right)\right)^{1/(p-1)-1}(1+o(1)).$$ Thus we obtain
\begin{align}\label{eq:11}
\sum_{n\le x}f_p(n)&=C_px(\log x)^{\frac{1}{p-1}-1}\sum_{n\le x}\frac{h_p(n)}{n}\left(1-\frac{\log n}{\log x}\right)^{\frac{1}{p-1}-1}c(x/n),
\end{align}
where $c(x/n)$ goes to $1$ as $x/n\to\infty$. We set some notation. Let
\begin{align*}
S(x)&=\sum_{n\le x}\frac{h_p(n)}{n}\left(1-\frac{\log n}{\log x}\right)^{\frac{1}{p-1}-1} c(x/n),\\
\alpha&=\frac{1}{p-1}-1,\\
u_n&=\frac{\log n}{\log x}\in [0,1],\\
w(u)&=(1-u)^{\alpha}.
\end{align*}
With this notation, we have $$S(x)=\sum_{n\le x}h_p(n)w(u_n) c(x/n).$$

Let $\varepsilon\in (0,1)$. On the interval $[0,1-\varepsilon]$, the function $w(u)\in C^\infty$ and by Newton's binomial theorem, we have
\begin{equation}\label{stimaw}
w(u)=(1-u)^{\alpha}=\sum_{m=0}^\infty{\alpha\choose m}(-u)^m=1-\alpha u+R_2(u),
\end{equation}
where the remainder $R_2(u)$ can be written in the form
$$R_2(u)=\frac{w^{(2)}(\zeta)}{2},$$
for a suitable $\zeta \in (0,u)$. We have
\begin{align*}
w^{(2)}(t)&=\alpha(\alpha-1)(1-t)^{\alpha-2}.
\end{align*}
This shows that for $t\in [0,1-\varepsilon]$, we have
\begin{align*}
\alpha(\alpha-1)\leq w^{(2)}(t)\le \alpha(\alpha-1)\varepsilon^{\alpha-2}.
\end{align*}
Hence, for all $u\in [0,1-\varepsilon]$, we have
\begin{equation*}
0\leq R_2(u)\le C_\varepsilon u^2,
\end{equation*}
where $C_\varepsilon$ is $\frac{\alpha(\alpha-1)}{2}\varepsilon^{\alpha-2}$.

Now, write
\begin{align*}
S_{main}(x)&=\sum_{n\le x^{1-\varepsilon}}\frac{h_p(n)}{n}w(u_n)c(x/n),&&S_{tail}(x)&=\sum_{x^{1-\varepsilon}<n\le x}\frac{h_p(n)}{n}w(u_n)c(x/n).
\end{align*}
For $n\le x^{1-\varepsilon}$, we have $u_n=\frac{\log n}{\log x}\le 1-\varepsilon$ and hence we may use (\ref{stimaw}). We deduce
\begin{align*}
S_{main}=&\sum_{n\le x^{1-\varepsilon}}\frac{h_p(n)c(x/n)}{n}
-\alpha \sum_{n\le x^{1-\varepsilon}}\frac{h_p(n)u_nc(x/n)}{n}\\
&+
\sum_{n\le x^{1-\varepsilon}}\frac{h_p(n)R_2(u_n)c(x/n)}{n}.
\end{align*}
We have
\begin{align*}
\sum_{n\le x^{1-\varepsilon}}\frac{h_p(n)u_nc(x/n)}{n}&=\frac{1}{\log x}\sum_{n\le x^{1-\varepsilon}}\frac{h_p(n)\log n c(x/n)}{n}.
\end{align*}
Analogously, for the remainder $R_2(u_n)$, we have
\begin{align*}
\sum_{n\le x^{1-\varepsilon}}\frac{h_p(n)R_2(u_n)c(x/n)}{n}&\le
C_\varepsilon\sum_{n\le x^{1-\varepsilon}}\frac{h_p(n)u_n^2c(x/n)}{n}\\
&=\frac{C_\varepsilon}{(\log x)^2}
\sum_{n\le x^{1-\varepsilon}}\frac{h_p(n)(\log n)^2c(x/n)}{n}.
\end{align*}
Now, observe that the series $\sum_{n}h_p(n)(\log n)^2/n$ converges because the support of $h_p(n)$ is on squareful integers. Therefore, the argument above shows that
\begin{align*}
S_{main}(x)&=\sum_{n\le x^{1-\varepsilon }}\frac{h_p(n)c(x/n)}{n}+O_\varepsilon\left(\frac{1}{\log x}\right).
\end{align*}

We now deal with $S_{tail}(x)$. We have $1-u_n=1-\frac{\log n}{\log x}\ge 0$ and $w(u_n)=(1-u_n)^\alpha\le \varepsilon^\alpha$. Thus
\begin{equation*}
S_{tail}(x)\le \varepsilon^{\alpha}\sum_{x^{1-\varepsilon}<n\le x}\frac{h_p(n)c(x/n)}{n}\le \varepsilon^\alpha \kappa \sum_{\substack{n>x^{1-\varepsilon }\\ n \textrm{ squareful}}}\frac{1}{n},
\end{equation*}
where $\kappa$ is the maximum of the function $c(x/n)$. From a theorem of Erd\H{o}s and Szekeres~\cite{es}, we have
$$|\{n\le y\mid n\textrm{ squareful }\}|\le C'y^{1/2}.$$
From this, it follows that
$$\sum_{\substack{n>y\\ n\textrm{ squareful}}}\frac{1}{n}\le C''y^{-1/2},$$
for some positive constant $C''$. Applying this with $y=x^{1-\varepsilon}$, we deduce
\begin{equation*}
S_{tail}(x)\le \varepsilon^{\alpha}\kappa C''x^{-(1-\varepsilon)/2}.
\end{equation*}
In other terms,
$$S_{tail}(x)\in O_\varepsilon(x^{-(1-\varepsilon)/2}).$$

Putting everything together, we get
$$S(x)=\sum_{n\le x}\frac{h_p(n)c(x/n)}{n}+O\left(\frac{1}{\log x}\right).$$

We now deal with $c(x/n)$. Recall that $\kappa$ is the maximum of the function $c(y)$. Let $t=\sum_{n}h_p(n)/n$. For every $\varepsilon>0$, let $\delta_\varepsilon'>0$ be such that $|c(x/n)-1|\le \varepsilon/(2t) $ when $x/n\ge \delta_\varepsilon'$ and let $\delta_\varepsilon''$ such that $$\sum_{\substack{n\, \mathrm{ squareful}\\ n> \delta_\varepsilon''}}\frac{h_p(n)}{n}\le \frac{\varepsilon}{2\kappa}.$$
Let $\delta_\varepsilon=\delta_\varepsilon'\delta_\varepsilon''$. When $x\ge \delta_\varepsilon$, 
we have
\begin{align*}
\left|\sum_{n\le x}\frac{h_p(n)c(x/n)}{n}-\sum_{n\le x}\frac{h_p(n)}{n}\right|&\le\sum_{\delta_\varepsilon'' < n}
\frac{h_p(n)\kappa }{n}+
\sum_{n\le \min\{x,\delta_\varepsilon''\}}
\frac{h_p(n)(c(x/n)-1) }{n}\\
&=\frac{\varepsilon}{2}+\sum_{x/n\ge \delta_\varepsilon'}
\frac{h_p(n)(c(x/n)-1) }{n}\le\varepsilon.
\end{align*}
We deduce
\begin{align*}
\lim_{x\to\infty}S(x)&=\sum_{n\ge 1}\frac{h_p(n)}{n}.
\end{align*}
Now the proof follows from~\eqref{eq:11}.
\end{proof}
\begin{remark}\label{remark1}{\rm
The proof of Theorem~\ref{thrms} shows that $$c_p=C_p\cdot\sum_{n}\frac{h_p(n)}{n},$$
where $C_p$ is the constant in Theorem~\ref{landauu}. Therefore, in principle $c_p$ can be calculated because we have the Euler product
$$\sum_{n}\frac{h_p(n)}{n}=\prod_{q\,\mathrm{prime}}\sum_{i=0}^\infty \frac{h_p(q^i)}{q^i}.$$
Thus 
\begin{align*}
\frac{c_p}{C_p}=\prod_{\substack{k\mid p-1\\ k\ne 1}}\prod_{q \textrm{ prime}}\left(1+q^{-{\bf o}_p(q)}+q^{-2{\bf o}_p(q)}+\cdots\right)=\prod_{\substack{k\mid p-1\\k\ne 1}}\prod_{q \textrm{ prime}}\left(1-q^{-{\bf o}_p(q)}\right)^{-1}.
\end{align*}
}
\end{remark}
\section{Sylow numbers arising from alternating groups}\label{sec:alt}

Let $n$ be a positive integer and let $p$ be an odd prime number.
Let $n=a_0+a_1p+\cdots+a_\ell p^\ell$ be the $p$-adic expansion of $n$, that is, $a_i\in \mathbb{N}$, $0\le a_i\le p-1$ and $a_\ell\ne 0$. From the structure of a Sylow $p$-subgroup $P$ of $\mathrm{Sym}(n)$, we see that the normalizer ${\bf N}_{\mathrm{Sym}(n)}(P)$ of $P$ in $\mathrm{Sym}(n)$ has order
$$\prod_{i=0}^\ell a_i!\left((p-1)^ip^{\frac{p^i-1}{p-1}}\right)^{a_i}.$$
Moreover $P\leq \mathrm{Alt}(n)$ and $|\mathrm{Alt}(n):{\bf N}_{\mathrm{Alt}(n)}(P)|=
		|\mathrm{Sym}(n):{\bf N}_{\mathrm{Sym}(n)}(P)|$.
As $a_i!\le (p-1)^{a_i}$, we deduce
 \begin{align*}
 |{\bf N}_{\mathrm{Sym}(n)}(P)|&\le (p-1)^{\sum_{i=0}^\ell a_i} (p-1)^{\sum_{i=0}^\ell ia_i}p^{\frac{\sum_{i}a_i(p^i-1)}{p-1}}\\
 &< p^{\sum_i(a_i+ia_i+a_i\frac{p^i-1}{p-1})}.
 \end{align*}
Since $p\ge 3$, it is straightforward to verify that
$$1+i+\frac{p^i-1}{p-1}\le p^i,$$
for every $i$. Thus 
\begin{align*}
 |{\bf N}_{\mathrm{Sym}(n)}(P)|&\le p^{\sum_ia_ip^i}=p^n.
 \end{align*}
We get
\begin{equation}\label{eq:indexnor}
|\mathrm{Sym}(n):{\bf N}_{\mathrm{Sym}(n)}(P)|\ge \frac{n!}{p^n}\ge\frac{(n/e)^n}{p^{n}}\ge\left(\frac{n}{ep}\right)^n.
\end{equation}

%Let $$\mathcal{A}=\{|\mathrm{Sym}(n):{\bf N}_{\mathrm{Sym}(n)}(P)|\mid n\in\mathbb{N}, P\,\,\mathrm{Sylow}\, p\textrm{-subgroup of }\mathrm{Sym}(n)\}$$and let $a(x)=|\{a\in\mathcal{A}\mid a\le x\}|$. From above, if $|\mathrm{Sym}(n):{\bf N}_{\mathrm{Sym}(n)}(P)|\le x$, then $n\log(n/ep)\le \log x$. 

\begin{lemma}\label{l:alt}\label{indexnor}
	Let $p$ be an odd prime, let $x$ be a real number and let $n$ be a positive integer with $n\le x$. If $|\mathrm{Sym}(n):{\bf N}_{\mathrm{Sym}(n)}(P)|\le x$, then $n\le e\log x$.
\end{lemma}
\begin{proof}
 Assume first $n\ge e^2p$. Then $\log (n/ep)\ge \log e=1$ and hence 
 \begin{equation*}e\log x > \log x\ge n\log(n/ep)\ge n.
 \end{equation*}
 
Assume next $n<e^2p$. It can be checked directly that $n\le e\log x$,  when $p=3$ and $n\le 22$, or $p=5$ and $n\le 36$, or $p=7$ and $n\le 51$. So we may assume $p>7$ and consequently $n <p^2.$
Thus $n=a_0+a_1p$. In this case, we may refine the inequality~\eqref{eq:indexnor}. Indeed,
\begin{align*}
|\mathrm{Sym}(n):{\bf N}_{\mathrm{Sym}(n)}(P)|&=\frac{(a_0+a_1p)!}{a_0!a_1!((p-1)p)^{a_1}}\ge \frac{(a_1p)^{a_1p}}{a_1!p^{2a_1}}\ge p^{a_1(p-2)}.
\end{align*}
If $|\mathrm{Sym}(n):{\bf N}_{\mathrm{Sym}(n)}(P)|\le x$, then $a_1(p-2)\log p\le \log x$. We have
\begin{align*}
a_1(p-2)\log p&=a_0+a_1p+a_1p\left(\left(1-\frac{2}{p}\right)\log p-1\right)-a_0\\
&=n+a_1p\left(\left(1-\frac{2}{p}\right)\log p-1\right)-a_0.
\end{align*}
In particular, $n\le \log x$, provided $a_1p((1-2/p)\log p-1)-a_0\ge 0.$  On the other hand, if 
$a_1p\left(\left(1-\frac{2}{p}\right)\log p-1\right)-a_0<0,$
then $p((1-2/p)\log p-1)<a_0/a_1\le p-1$, that is, $(1-2/p)\log p\le 2-1/p$. It is easy to verify that this inequality is satisfied only when $p\le 7$.
\end{proof}

\section{Sylow numbers arising from simple groups of Lie type}\label{sec:Lie}
The scope of this section is to prove an analogue of Lemma~\ref{l:alt} for simple groups of Lie type.

For this section and the next, if $A\subseteq \mathbb{N}$ and $x>0$, we set
$$A(x)=|\{a\in A: a\le x\}|.$$

\begin{lemma}\label{l:Lie}
Let $p$ be an odd prime and let $A_{Lie}$ be the set consisting of the positive natural numbers $n$ such that there exists a non-abelian simple group of Lie type  having $n$ Sylow $p$-subgroups. There exists a positive constant $c_{Lie}$ such that $A_{Lie}(x)\le c_{Lie}x^{1/2}/\log x$ for every $x\ge 2$.
\end{lemma}
\begin{proof}
Let $m\in A_{\mathrm{Lie}}$. Thus there exists a non-abelian simple group
of Lie type $K={}^d\Sigma_r(q)$ such that
$m=|K:{\bf N}_K(P)|\le x$. Here we use the notation of~\cite{GLS}, so
$d\in\{1,2,3\}$, $\Sigma\in\{A,B,C,D,E,F,G\}$ describes the type, and $r$
is the Lie rank of $K$.

Let $m(K)$ denote the minimal degree of a permutation representation of
$K$. The value of $m(K)$ is explicitly known in the literature. We use
\cite[Table~4]{PS}, since this table takes into account some mistakes
present in earlier tabulations of this value. We deduce that there exists
a positive constant $\kappa$ such that
$m(K)\ge \kappa q^{r}$.

Assume first that $r\ge 2$, that is, $K$ is not $\mathrm{PSL}_2(q)$. Thus $\kappa q^r\le m(K)\le x$. We let $N(x)$ denote the number of solutions $(r,q)$ to the inequality $\kappa q^r\le x$, where $q$ is a prime power and $r\ge 2$. We claim that there exists a constant $C$ such that
\begin{equation}\label{value N}
N(x)\le C\frac{x^{1/2}}{\log x},
\end{equation}
for every $x>1$.

Let $A(y)$ denote the number of prime powers $\le y$, so
\[
A(y)=\sum_{k\ge 1} \pi\!\left(y^{1/k}\right),
\]
where $\pi(t)$ is the prime counting function.  Using the standard bound
\[
\pi(t)\le C_1\,\frac{t}{\log t}\qquad(t\ge 2),
\]
we first isolate the $k=1$ term:
\[
\pi(y) \le C_1\,\frac{y}{\log y}.
\]

For $k\ge 2$ we avoid estimating $\log(y^{1/k})$ directly.  
Since $\pi$ is increasing, we have
\[
\pi\!\left(y^{1/k}\right) \le \pi(\sqrt{y})\qquad(k\ge 2),
\]
and $\pi(y^{1/k})=0$ once $y^{1/k}<2$, which occurs for all
$k > \log_2 y$.  Hence
\[
\sum_{k\ge 2}\pi\!\left(y^{1/k}\right)
\;\le\;
\sum_{2\le k\le \log_2 y} \pi(\sqrt{y})
\;\le\;
(\log_2 y)\,\pi(\sqrt{y}).
\]

Applying the prime number bound at $t=\sqrt{y}$ gives
\[
\pi(\sqrt{y})
\;\le\;
C_1\,\frac{\sqrt{y}}{\log \sqrt{y}}
=
\frac{2C_1\sqrt{y}}{\log y}.
\]

Therefore
\[
A(y)
=
\sum_{k\ge 1}\pi\!\left(y^{1/k}\right)
\le
C_1\,\frac{y}{\log y}
+
\frac{2C_1\sqrt{y}\,\log_2 y}{\log y}.
\]

Since
\[
\frac{2C_1\sqrt{y}\,\log_2 y}{\log y}
= o\!\left(\frac{y}{\log y}\right)\qquad (y\to\infty),
\]
there exists $y_0$ and $C_2>0$ such that for all $y\ge y_0$,
\begin{equation}\label{eq:Aybound}
	A(y)\le C_2\,\frac{y}{\log y}.
\end{equation}

Now consider $N(x)$. For fixed $r\ge 2$, the condition $\kappa q^r\le x$ implies $q\le (x/\kappa)^{1/r}$, hence by \eqref{eq:Aybound},
\[
N(x)
\le \sum_{r=2}^{\lfloor \log_2 (x/\kappa)\rfloor}
A\!\left((x/\kappa)^{1/r}\right)
\le \sum_{r=2}^{\lfloor \log_2 (x/\kappa)\rfloor}
C_2\,\frac{(x/\kappa)^{1/r}}{\log((x/\kappa)^{1/r})}.
\]
Since $\log(x^{1/r})=\frac{1}{r}\log x$, we may rewrite this as
\[
N(x)
\le
\frac{C_2}{\log (x/\kappa)}
\sum_{r=2}^{\lfloor \log_2 (x/\kappa)\rfloor}
r\,(x/\kappa)^{1/r}.
\]

The dominant contribution arises from $r=2$:
\[
\text{(term with $r=2$)}
=
\frac{2C_2\,(x/\kappa)^{1/2}}{\log (x/\kappa)}.
\]
For $r\ge 3$, we have $(x/\kappa)^{1/r}\le (x/\kappa)^{1/3}$ and
\[
\sum_{r=3}^{\lfloor \log_2 (x/\kappa)\rfloor} r
\le \frac{1}{2}(\log_2 (x/\kappa))^2,
\]
so
\[
\sum_{r=3}^{\lfloor \log_2 (x/\kappa)\rfloor}
\frac{C_2 r\,(x/\kappa)^{1/r}}{\log x}
\le
\frac{C_2 (x/\kappa)^{1/3}}{\log (x/\kappa)}
\cdot \frac{1}{2}(\log_2 (x/\kappa))^2
= \frac{C_2}{2}\,(x/\kappa)^{1/3}\log (x/\kappa).
\]
Since $x^{1/3}\log x = o\!\left(\dfrac{x^{1/2}}{\log x}\right)$
as $x\to\infty$, this contribution is of lower order.  This finally proves our claim~\eqref{value N}.

In particular, if $r\ge 2$, then we have at most $3\cdot 7\cdot N(x)\le 21 Cx^{1/2}/\log x$ choices for $K$, because $3$ accounts for the possibilities of $d$ and $7$ accounts for the possibilities of $\Sigma$ in $K={^d\Sigma}_r(q)$.

Finally, assume $r=1$. In this case, if $\gcd(p,q)=1$, then the number of Sylow $p$-subgroups is $q(q\pm 1)/\gcd(2,q-1)$. Since this number is also a linear factor of $q^2$, we also see that the contribution of these groups has order of magnitude $Cx^{1/2}/\log x$.
If $p\mid q$, then the number of Sylow $p$-subgroups in $\mathrm{PSL}_2(q)$ is $1+q$. As $p\mid q$, we have $q=p^\ell$, for some positive integer $\ell$. As $1+q\le x$, we deduce that $\ell\le\log x$. Therefore the contribution of these groups is negligible.
\end{proof}

\section{Proof of Theorem~$\ref{thrm}$}
In this section, we combine the results in the previous sections and we prove Theorem~\ref{thrm}. The proof will follow from a technical proposition. 

\begin{proposition}\label{prop}
Let $A\subset\mathbb{N}$ satisfy
$A(x)\ \le cx(\log x)^{\alpha}$ when $x\ge 2$,
for some constants $c>0$ and $-1<\alpha<0$.  Let $X\subset\mathbb{N}$ satisfy \[X(x) \le
c'\,\frac{x^{1/2}}{\log x}
\qquad(x\ge 2),
\]
and let $B$ be the multiplicative semigroup generated by $X$.  Set
$$C=A\cdot B=\{ab:\ a\in A,\ b\in B\}.$$
Then
\[
C(x)\ \ll\ x(\log x)^{\alpha}
\qquad\text{as }x\to\infty,
\]
and in particular $C$ has natural density $0$.
\end{proposition}

\begin{proof}
We write $X=\{x_1<x_2<\cdots\}$.  
From $n=X(x_n)\le c'x_n^{1/2}/\log x_n$ we obtain the inequality
\[
x_n \ \ge\ \frac{n^2}{c'^2}(\log x_n)^2.
\]
For a suitable $n_0,$ for every $n\geq n_0$, we have $\log x_n\ge \frac12\log n$, hence
\[
x_n \ \ge\ \frac{n^2(\log n)^2}{4c'^2}
\qquad(n\ge n_0).
\]
Consequently,
\begin{equation}\label{eq:tailX}
\sum_{n\ge n_0}\frac{1}{x_n}
\ \le\ 
4c'^2\sum_{n\ge n_0}\frac{1}{n^2(\log n)^2}
<\infty.
\end{equation}
Since $B$ is the multiplicative semigroup generated by $X$, we have
\[
\sum_{b\in B}\frac{1}{b}
=
\prod_{n=1}^{\infty}\left(1+\frac1{x_n}+\frac1{x_n^2}+\cdots\right)
=
\prod_{n=1}^{\infty}\frac{1}{1-1/x_n},
\]
and for $n$ large the assumption $x_n\to\infty$ implies $\frac {1}{x_n}\le \tfrac12$.
Using $\log\frac{1}{1-t}\le 2t$ for $0<t\le1/2$ and \eqref{eq:tailX}, we obtain
\[
\log\left(\sum_{b\in B}\frac{1}{b}\right)
=
\sum_n \log\left(\frac{1}{1-1/x_n}\right)
\le
2\sum_{n}\frac{1}{x_n}
<\infty.
\]
Hence 
\begin{equation}\label{eq:SB}
\sum_{b\in B}\frac{1}{b} <\infty.
\end{equation}

Every $m\in C$ with $m\le x$ has a representation $m=ab$ with
$a\in A$, $b\in B$, and $b\le x$.  Thus
\[
C(x)
\le 
\sum_{\substack{b\in B\\ b\le x}}A\!\left(\frac{x}{b}\right)
\le
\sum_{\substack{b\in B\\ b\le x}}
c\,\frac{x/b}{(\log(x/b))^{-\alpha}}
=
c\,x\sum_{\substack{b\in B\\ b\le x}}
\frac{1}{b\,(\log(x/b))^{-\alpha}}.
\]
Write $\beta=-\alpha\in(0,1)$.  Split the sum at $b=x^{1/2}$:

\[
C(x)=C_1(x)+C_2(x)
\]
with
\[
C_1(x)
=c\,x\sum_{\substack{b\in B\\ b\le x^{1/2}}}
\frac{1}{b(\log(x/b))^{\beta}},
\qquad
C_2(x)
=c\,x\sum_{\substack{b\in B\\ x^{1/2}<b\le x}}
\frac{1}{b(\log(x/b))^{\beta}}.
\]

If $b\le x^{1/2}$, then $\log(x/b)\ge\tfrac12\log x$, hence
\begin{equation}\label{eq:clubsuits}
C_1(x)
\le
c\,x\,\frac{2^{\beta}}{(\log x)^{\beta}}
\sum_{b\in B}\frac{1}{b}
= c\,2^{\beta}\,\frac{x}{(\log x)^{\beta}}\sum_{b\in B}\frac 1b
= c\,2^{-\alpha}\,\,x(\log x)^{\alpha}\sum_{b\in B}\frac 1b.
\end{equation}
For $b>x^{1/2}$ we use $(\log(x/b))^{-\beta}\le1$ and obtain
\[
C_2(x)
\le
c\,x\sum_{\substack{b\in B\\ b>x^{1/2}}}\frac{1}{b}
=
c\,x\cdot T(x),
\qquad\textrm{where }
T(x)=\sum_{\substack{b\in B\\ b>x^{1/2}}}\frac{1}{b}.
\]
Since $\sum_{b\in B}1/b<\infty$ by~\eqref{eq:SB}, the tail $T(x)\to0$ as $x\to\infty$.
Therefore,
\begin{equation}\label{eq:spadesuit}
\frac{C_2(x)}{x(\log x)^{\alpha}}
=
c\,T(x)\,(\log x)^{-\alpha}
\longrightarrow 0
\qquad(x\to\infty),
\end{equation}
because $-\alpha>0$.

Combining~\eqref{eq:tailX},~\eqref{eq:SB}, \eqref{eq:clubsuits} and
\eqref{eq:spadesuit},
for all sufficiently large $x$ we obtain
\[
C(x)
\le
c\,2^{-\alpha}\,\sum_{b\in B}\frac 1b\,x(\log x)^{\alpha}
+ o\bigl(x(\log x)^{\alpha}\bigr)
\qquad(x\to\infty).
\]
In particular,
\[
C(x)\ \ll\ x(\log x)^{\alpha},
\qquad
\frac{C(x)}{x}\ll (\log x)^{\alpha}\xrightarrow[x\to\infty]{}0,
\]
so $C$ has natural density $0$.
\end{proof}
\begin{proof}[Proof of Theorem~$\ref{thrm}$]The proof follows immediately from the main result of~\cite{mhall} and
Proposition~\ref{prop}, applied with $A$ the set of Sylow $p$-numbers of
solvable groups and $X$ the set of Sylow $p$-numbers of non-abelian
simple groups. The hypotheses of Proposition~\ref{prop} concerning $A(x)$
are satisfied by Theorem~\ref{thrms}, and the hypotheses of
Proposition~\ref{prop} concerning $X(x)$ are satisfied by
Lemmas~\ref{l:alt} and~\ref{l:Lie}.
\end{proof}

\end{document}